\documentclass{article}

\usepackage{amsmath}
\usepackage{amssymb}
\usepackage{a4}
\usepackage{xspace,amscd}

\def \dim{\mathop{\rm dim}\nolimits}

\def \M{{\bf m}} 
\def \R{{\mathbb R}}
\def \C{{\mathbb C}}

\def \N{{\mathbb N}}

\def \AA{{\cal A}}

\def \EE{{\cal E}}

\def \KK{{\cal{K}}}

\newtheorem{theorem}{Theorem}[section]
\newtheorem{corollary}[theorem]{Corollary}
\newtheorem{proposition}[theorem]{Proposition}

\newtheorem{remark}[theorem]{Remark}
\newtheorem{remarks}[theorem]{Remarks}
\newtheorem{example}[theorem]{Example}

\newtheorem{definition}[theorem]{Definition}
\newtheorem{conjecture}[theorem]{Conjecture}

\newenvironment{proof}
 {\begin{trivlist} \item[\hskip \labelsep {\bf Proof.}]}
 {\hfill$\Box$\end{trivlist}}
\newenvironment{lproof}[1]
 {\begin{trivlist} \item[\hskip \labelsep {\bf Proof (#1).}]}
 {\hfill$\Box$\end{trivlist}}

\def\Gp{\varphi }
\def\K{\mathbb{K}}
\def\X{\xi }
\def\E{\eta }
\def\en{\end{eqnarray*}}
\def\lt{\left}
\def\rt{\right}
\def\fr{\frac}
\def\pr{\partial}

\def\Gth{\theta}

\def\Ga{\alpha}
\def\Gb{\beta}

\def\Gg{\gamma }
\def\<{\langle}

\def\M{\mathfrak{m}}

\def\LT{\mathop{\rm LT}\nolimits}
\def\id{\mathop{\rm id}\nolimits}
\def\cod{\mathop{\rm cod}\nolimits}
\def\mod{\mathop{\rm{\,mod}}\nolimits}

\def\dwi{\dfrac{\partial W_1}{\partial y}}
\def\dwii{\dfrac{\partial W_1}{\partial y}}

\def\derlog{{\rm{Derlog}}}

\begin{document}
\renewcommand{\theenumi}{(\roman{enumi})}
\normalsize

\title{Vector fields liftable over corank $1$ stable maps} 
\author{Kevin Houston and Daniel Littlestone\\
School of Mathematics \\ University of Leeds \\ Leeds, LS2 9JT, U.K. \\
e-mail: k.houston@leeds.ac.uk \\
http://www.maths.leeds.ac.uk/$\sim $khouston/
}
\date{\today }
\maketitle
\begin{abstract}
Generators for the module of vector fields liftable over corank $1$ stable complex analytic maps from an $n$-manifold to an $n+1$-manifold are found.  This is applied to the classification of the singularities occuring in generic one-parameter families of maps between these spaces.


\end{abstract}

\section{Introduction}
In his influential paper \cite{wavefront} Arnol'd showed the importance of liftable vector fields in the study of singularities, in particular for wave fronts. Let $\K $ denote the field of real numbers or complex numbers and 
suppose that $f:\K ^n\to \K ^p$ is a smooth map.  A vector field $\X $ on $\K ^p$ is {\em{liftable}} if there exists a vector field $\eta $ on $\K ^n$ such that $df \circ \eta = \X \circ f$ where $df$ is the differential of $f$. These vector fields form a module over the ring of functions on $\K ^p$.
These are useful in the classification of singularities of mappings since one can integrate them to produce diffeomorphisms that preserve the properties of the map.

Arnol'd showed that for a map $\varphi :(\C ^n,0)\to (\C ^n,0)$ given by $\varphi (x_1,\dots , x_{n-1}, y)= (x_1,\dots , x_{n-1}, y^{n+1}+\sum_{i=1}^{n-1} x_iy^i)$ the module of liftable vector field is generated freely by $n$ vector fields.
For finitely $\AA $-determined complex analytic map germs $f:(\C ^n,0)\to (\C ^p,0)$ with $n\geq p$, the discriminant of $f$ is a hypersurface in $\C ^p$. The vector fields tangent to this hypersurface coincide with the liftable vector fields. 
Corollary~6.13 of \cite{looijenga}, following Bruce \cite{b2} and Zakalyukin \cite{Z}, shows that for a stable $f$ the vector fields tangent to the discriminant are generated freely by $p$ generators.

If the vector fields tangent to a hypersurface are freely generated, then the hypersurface is called a {\em{free divisor}}, \cite{saito}. The theory of free divisors is particularly good and there are many examples areas in which they occur naturally, such as in discriminants \cite{looijenga}, bifurcation theory \cite{bf,fss,gs,montaldi}, and hyperplane arrangements, \cite{ot}.

However, in the case of maps with $n<p$, where the discriminant of $f$ is the image, little is known about the liftable vector fields. We can see the problem in the case of $p=n+1$ where the image is a hypersurface. Even in this nice case the module of tangent vector fields is not a free module. 
For the cross cap, $\varphi : (\K ^2,0)\to (\K ^3,0)$ given by $\varphi (x,y)=(x,y^2,xy)$, the module of tangent vector fields is generated by {\em{four}} vector fields rather than three, which is the dimension of the codomain. This unfortunate state of affairs has meant that the theory of tangent vector fields for $n<p$ has been somewhat negligible in comparison with the $n\geq p$ case. 
There have been attempts to circumvent the problem. For example, Damon introduced the notion of Free* divisor, see \cite{damonfree}.

The only general results known to the authors for liftable and tangent vector fields for $n<p$ are in \cite{hm}: For a stable corank $1$ map from $(\C ^n,0)$ to $(\C ^{n+1},0)$ of multiplicity $k$ the module of tangent vector fields has projective dimension $1$ (and hence is not free) and the number of generators is $3k-2$. In the preamble and proof of Lemma 3.5 they explain a theory which, in effect, gives an algorithm for finding these generators but it is computationally labour intensive and the authors of \cite{hm} do not use it to explicitly calculate any vector fields.

In this paper, for the normal form of minimal stable corank $1$ maps from $(\K ^n,0)$ to $(\K ^{n+1},0)$ of multiplicity $k$, we describe $3k-2$ smooth liftable vector fields and show that, over the ring of complex analytic functions, these are generators of the module of complex analytic tangent vector fields. 
The normal form of a stable corank $1$ map from $(\K ^n,0)$ to $(\K ^{n+1},0)$ of multiplicity $k$ was first given by Morin in \cite{morin} and is of the form $\Gp_k:(\K^{2k-2},0)\to(\K^{2k-1},0)$ with
\begin{eqnarray*}
&&\Gp_k(u_1,\ldots ,u_{k-2},v_1,\ldots ,v_{k-1},y)\\
&&\qquad\qquad\qquad\qquad=\left(u_1,\ldots ,u_{k-2},v_1,\ldots ,v_{k-1},y^k+\sum_{i=1}^{k-2}u_iy^i,\sum_{i=1}^{k-1}v_iy^i\right)
\end{eqnarray*}
The case of $k=2$ gives the cross cap.

The second author found the generators by first using the computer package {\texttt{Singular}} (see \cite{singular}) to find the fields for low values of $k$ (computing power restricted this to $k<7$). Then by trial and error for each vector field $\xi $ he found a vector field in the domain $\eta $ so that $d\Gp _k\circ \eta = \xi \circ \Gp_k$. The results were then generalized, again with some trial and error, to the case of arbitrary $k$. (That the fields are lowerable has been verified via computer for $k<7$). The first author then proved that these vector fields
along with the Euler vector field generate the module of liftable vector fields in the case of $\K =\C$.

There are three families of liftable vector fields, each containing $k-1$ vector fields.  Using {\texttt{Singular}} it was conjectured that a field $\xi $ from any of the three families is such that $\xi (h)=0$ where $h$ is the defining equation of the image of $\varphi _k$.

The paper is arranged as follows. Section~\ref{sec:defs} gives the basic definitions such as corank $1$, stable, liftable, lowerable and so on. In Section~\ref{sec:main} we state the main results. We give three families of vector fields and show that each member is liftable, Theorems~\ref{firstfamily}, \ref{secondfamily}, and \ref{thirdfamily}. In Corollary~\ref{cor:generate} we show that, over the complex field, these families and the Euler vector field generate the module of tangent vector fields. In the next section it is proved that the vector fields are liftable.

Since we claim that liftable vector fields are useful we give in Section~\ref{sec:appl} an application of the results. We begin to classify linear functions on $\C ^p$ preserving the image of $\varphi _k$. This can be used to find corank $1$ maps from $\C ^n$ to $\C ^{n+1}$ of $\AA _e$-codimension $1$ (see Definition~\ref{defn:aecod}). These maps occur generically in one parameter families of maps and hence are of particular interest. 

In the proof of Theorem~7.3 of \cite{topaug}, following the ideas in \cite{damon1,cplxgen}, it is proved that a suitably generic linear function on the image of a trivial unfolding of a corank $1$ cross cap mapping can be used to produce an $\AA _e$-codimension $1$ map germ. (The converse is obviously true.) Here we reprove the result for {\em{minimal}} cross cap mappings (though note that the generalization to trivial unfoldings of these is not difficult) but in this case we can say precisely what is meant by `generic' in `generic linear function', see Corollary~\ref{cor:aecod1}. Thus we begin a classification similar to the one for the cross cap of multiplicity $2$ given in \cite{west}.

The second author thanks EPSRC for financial support during this work that led to his MPhil thesis, \cite{danthesis}.

\section{Definitions and standard results}
\label{sec:defs}

The definitions and results in this section are all known, see \cite{wall} for standard definitions of singularity theory and \cite{damonwark} for results on liftable and tangent vector fields.

In the intitial definitions in this section we shall mostly assume we are working with smooth (i.e., infinitely differentiable) maps, and will note that the theory for real analytic and complex analytic maps is similar. We shall have $\K =\R $ or $\K =\C $.
We recall the definition of $\AA $-equivalence. 
\begin{definition}
Two smooth map germs $f_1:(\K^n,0)\to(\K^p,0)$ and $f_2:(\K^n,0)\to(\K^p,0)$ are {\em{$\AA $-equivalent}} if there exist diffeomorphisms $\varphi$ and $\psi$ for which the following diagram commutes.
\[\begin{CD}
(\K^n,0) @>f_1 >> (\K^p,0)\\
@V{\Gp}VV @VV{\psi}V\\
(\K^n,0) @>f_2 >> (\K^p,0).
\end{CD}\]
\end{definition}
This is also known as {\em{Right-Left-equivalence}}.

The set of smooth (or $\K $-analytic) function germs $f:(\K^n,0)\to \K $ forms a ring, denoted by $\EE _{n}$, using pointwise addition and multiplication. We denote the tangent bundle of the germ $(\K ^n,0)$ by $T(\K ^n,0)$ and let $\pi _n:T(\K ^n,0)\to (\K ^n,0)$ be the natural projection.

For a map $f:(\K^n,0)\to(\K^p,0)$ we define the {\em{vector fields along $f$}}, denoted $\Gth(f)$, to be 
\[
\Gth(f)=\{\text{smooth maps }h:(\K^n,0)\to T(\K^p,0) \text{ such that }\pi_p\circ h=f\}.
\]

First note that $\theta (f)$ is a module over $\EE _n$ and $\theta (f)\cong \EE _n^p$, i.e., $\theta (f)$ is isomorphic to $p$ copies of $\EE _n$. Second note that for the identity map, $\id _n :(\K ^n,0)\to (\K ^n,0)$, we have that $\theta (\id _n)$ is isomorphic to the module of vector fields of $(\K ^n,0)$. We denote this latter by $\theta _n$. As $\theta _n\cong \EE _n^n$ we shall write a vector field on $(\K ^n,0)$ as an $n$-tuple of elements of $\EE _n$.

Consider the diagram
\[
\begin{CD}
T(\K^n,0) @> df >> T(\K^p,0)\\
@A{\E}AA @AA{\X}A\\
(\K^n,0) @>f >> (\K^p,0).
\end{CD}
\]
where $\E \in \theta _n$ and $\X \in \theta _p$ are vector fields. By the obvious compositions we can construct elements of $\theta (f)$. Therefore, consider the two maps $tf:\Gth_n\to\Gth(f)$, where $tf(\E)=df\circ\E$ and $wf:\Gth_p\to\Gth(f)$, where $wf(\X)=\X\circ f$. 

\begin{definition}
A map $f:(\K ^n,0)\to (\K ^p,0)$ is called {\em{stable}} if 
\[
\theta (f) =tf(\Gth_n)+wf(\Gth_p) .
\]
\end{definition}

\begin{remark}
Mather in \cite{matherv} showed that a map $f$ is stable in the above sense (which he called infinitessimally stable) if and only if there exists a neighbourhood in the space of smooth maps (with the Whitney topology) such that all maps in the neighbourhood are $\AA$-equivalent to $f$.
\end{remark}

\begin{definition}
We say a map $f:(\K ^n,0)\to (\K ^p,0)$ is {\em{corank $1$}} if the rank of the Jacobian matrix of $f$ at $0$ is equal to $\min(n,p)-1$, that is, it is one less than maximal.
\end{definition}
These are good maps to study as these are the first type of singular maps that one encounters (after studying immersions and submersions). 

\begin{example}
Let $f:(\K ^2,0)\to (\K ^3,0)$ be given by $f(x,y)=(x,y^2,xy)$. This is known as the {\em{cross cap}} or the {\em{Whitney Umbrella}}. This map is corank $1$ as one can verify by a quick calculation. The stable condition is shown by a more complicated calculation.
\end{example}
This mapping was generalized by Morin in \cite{morin} 
Here we shall restrict ourselves to the case where $p=n+1$.
\begin{definition}
For $k\geq 2$ the {\em{minimal cross cap mapping of multiplicity $k$}} is the map $\varphi _k:(\K^{2k-2},0)\to(\K^{2k-1},0)$ given by
\begin{eqnarray*}
&&\varphi _k(u_1,\ldots ,u_{k-2},v_1,\ldots ,v_{k-1},y)\\
&&\qquad\qquad\qquad\qquad=\left(u_1,\ldots ,u_{k-2},v_1,\ldots ,v_{k-1},y^k+\sum_{i=1}^{k-2}u_iy^i,\sum_{i=1}^{k-1}v_iy^i\right)
\end{eqnarray*}
We shall label the coordinates of the target $U_1,\ldots ,U_{k-2},V_1,\ldots ,V_{k-1},W_1$ and $W_2$, respectively. The sets of coordinates will be abbreviated to $\underline{U}$, $\underline{V}$ and $\underline{W}$ respectively.
\end{definition}
These maps are of interest because up to a trivial unfolding and up to $\AA $-equivalence they classify stable corank $1$ maps. (The precise statement is in the following theorem.) A trivial unfolding of a map $f:(\K ^n,0)\to (\K ^p,0)$ is a map $F:(\K ^{n+q},0)\to (\K ^{p+q},0)$ given by $F(x,z)=(f(x),z)$ where $x$ are coordinates of $\K ^n$ and $z$ are coordinates of $\K ^q$.
\begin{theorem}[\cite{morin}]
A germ $f:(\K ^n,0)\to (\K ^{n+1},0)$ is a stable corank $1$ germ if and only there exists a $k$ such that $f$ is $\AA $-equivalent to the trivial unfolding of the minimal cross cap mapping of multiplicity $k$.
\end{theorem}
In the case of $n=2$, (up to $\AA $-equivalence) the cross cap is the only singular stable mono-germ $f:(\K ^2,0)\to (\K ^3,0)$.

%
%

We now come to the main objects of study: liftable vector fields.
\begin{definition}
\label{df:df1}
Let $f$ be a smooth mapping $f:(\K^n,0)\to(\K^p,0)$. A vector field $\X$ on $(\K^p ,0)$ is {\em{liftable over $f$}} if there is a vector field $\E$ on $(\K ^n,0)$ such that $df\circ\eta=\X\circ f$. That is, the following diagram commutes
\[
\begin{CD}
T(\K^n,0) @> df >> T(\K^p,0)\\
@A{\E}AA @AA{\X}A\\
(\K^n,0) @>f >> (\K^p,0).
\end{CD}
\]
In these circumstances $\E$ is called {\em{lowerable}}.
\end{definition}

\begin{example}
\label{eg:wulift}
For the Whitney umbrella $\varphi _2(v_1,y)=(v_1,y^2,v_1y)$ the following are liftable vector fields:
\[
\lt(\begin{array}{c} W_2 \\ 0 \\ V_1W_1 \end{array}\rt),  \lt(\begin{array}{c} -V_1 \\ 2W_1 \\ 0 \end{array}\rt), \lt(\begin{array}{c} 0 \\ 2W_2 \\ V_1^2 \end{array}\rt) \quad \text{and} \quad\lt(\begin{array}{c} V_1 \\ 2W_1 \\ 2W_2 \end{array}\rt).
\]
The corresponding lowerable vector fields are (respectively)
\[
\lt(\begin{array}{c} v_1y \\ 0 \end{array}\rt),  \lt(\begin{array}{c} -v_1 \\ y \end{array}\rt) , \lt(\begin{array}{c} 0 \\ v_1 \end{array}\rt)  \quad \text{and} \quad\lt(\begin{array}{c} v_1 \\ y \end{array}\rt).
\]
This can be shown by composing the Jacobian with each lowerable. For example, taking the second vector field in the list we have
\[
d\Gp_2\circ  
\lt(\begin{array}{c} -v_1 \\ y \end{array}\rt)
=\lt(\begin{array}{cc} 1 & 0 \\ 0 & 2y \\ y & v_1 \end{array}\rt)\lt(\begin{array}{c} -v_1 \\ y \end{array}\rt)=\lt(\begin{array}{c} -v_1 \\ 2y^2 \\ 0 \end{array}\rt)=\lt(\begin{array}{c} -V_1 \\ 2W_1 \\ 0 \end{array}\rt)\circ\Gp_2.
\]
The liftability of the other vector fields is just as easily verified.

In fact, it can be shown that these vector fields generate the module of liftable vector fields, see \cite{west}.
\end{example}

The goal of this paper is to find a set of generators for the module of vector fields liftable over $\varphi _k$, the cross cap map of multiplicity $k$.
This is achieved for $\K =\C $ and for the module of polynomial liftable real vector fields.

\begin{definition}
\label{df:de}
A mapping $f:(\K^n,0)\to(\K^p,0)$ is said to be {\em{quasihomogeneous}} (or {\em{weighted homogeneous}}) of type $(w_1,\ldots,w_n;d_1,\ldots,d_p)$, with $w_i,d_j\in\N \cup \{ 0 \} $ if the relation
\[f_j(t^{w_1}x_1,\ldots,t^{w_n}x_n)=t^{d_j}f_j(x_1,\ldots,x_n)\]
holds for each coordinate function $f_j$ of $f$ for all $t\in(\K,0)$. The number $w_i$ is called the {\em{weight}} of the variable $x_i$ and the number $d_j$ is the {\em{degree}} of the function $f_j$. 

Let $X_1,\dots ,X_p$ denote the coordinates on $\K ^p$. Then, the vector field given by
\[
\X_e=
\lt(\begin{array}{c} d_1 X_1 \\ \vdots \\ d_p X_p \end{array}\rt)
\]
is called the {\em{Euler vector field}} and is denoted by $\X_e$.
\end{definition}

\begin{proposition}
\label{prop:eulerlift}
If the map $f$ is quasihomogeneous, then the Euler vector field $\X_e$ is liftable over $f$.
\end{proposition}
\begin{proof}
Differentiating both sides of the relationship 
\[
f_j(t^{w_1}x_1,\ldots,t^{w_n}x_n)=t^{d_j}f_j(x_1,\ldots,x_n)
\]
and setting $t=1$ we find that 
\[
\sum_{i=1}^nw_ix_i\fr{\pr f_j}{\pr x_i}=d_jf_j .
\]
Since $X_j\circ f=f_j$ we deduce that
\[
J_f \lt(\begin{array}{c}w_1x_1 \\ \vdots \\ w_n x_n \end{array}\rt) = 
\lt(\begin{array}{c} d_1 X_1 \\ \vdots \\ d_p X_p \end{array}\rt) \circ f ,
\]
where $J_f$ is the Jacobian of $f$. That is, the Euler vector field is liftable.
\end{proof}


\begin{example}
It is easy to show that the cross cap mapping $\varphi _k$ is quasihomogeneous and that its Euler vector field is
\begin{displaymath}\X_e=\lt(\begin{array}{c}
(k-1)U_1\\
(k-2)U_2\\
\vdots\\
2U_{k-2}\\
(k-1)V_1\\
(k-2)V_2\\
\vdots\\
V_{k-1}\\
kW_1\\
kW_2
\end{array}\rt)\end{displaymath}
\end{example}

%
%

Vector fields liftable over $f$ are closely related to vector fields tangent to the discriminant of $f$. (The discriminant is the image under $f$ of the points for which the rank of the differential is less than $p$ and hence is equal to the image for $n<p$.)

\begin{definition}
Suppose that $V$ is a $\K $-analytic variety defined by the ideal $I(V)$ in $(\K ^p,0)$. A vector field $\X \in \theta _p$ is said to be {\em{tangent to $V$}} if
\[
\X(I(V))\subseteq I(V) .
\]  
The module of such vector fields is denoted $\derlog (V)$.
\end{definition}

\begin{example}
The image of $\varphi _2:(\C ^2,0)\to (\C ^3,0)$ is given by $W_2^2-V_1^2W_1=0$. The vector fields of Example~\ref{eg:wulift} are tangent to this image as can easily be checked.
\end{example}
For stable maps with $\K =\C$ the notions of liftable and tangent to the discriminant are equivalent. This equivalence was proved in \cite{wavefront} and \cite{saito}, (see also \cite{b2}), for $n\geq p $, and in \cite{damonwark} for $n<p$.

This useful equivalence does not hold for real analytic maps because in this case the image of map may not be an analytic set. For example, the image of $\varphi _2:(\R ^2,0)\to (\R ^3,0)$, does satisfy the real version of the equation $W_2^2-V_1^2W_1=0$ given in the preceding example. However, this real version adds an extra `handle', the $V_1$ axis, and hence the name `umbrella'. To get the image of $\varphi _2$ we need the additional inequality $V_1\geq 0$ for example and hence the image is not analytic.

%
%
\section{Main results}
\label{sec:main}
In this section we state the main results. We shall show that for the minimal cross cap mapping of multiplicity $k$, in addition to the Euler vector field, we find three families, each consisting of $k-1$ elements, of liftable vector fields. Proofs are deferred to the next section.

In the case of $\K =\C$ we shall show that these families and the Euler vector fields in fact generate $\derlog (V)$ where $V$ is the image of $\varphi _k$.
Furthermore, again over $\C $, it is conjectured that the three families generate the vector fields $\xi $ such that $\xi (h)=0$ where $h$ is the defining equation of the image of $\varphi _k$.

We shall denote the members of the families by $\xi ^f_j$ where $1\leq f\leq 3$ and $1\leq j\leq k-1$.
This can be written in component form as 
\[
\xi ^f_j=
\left(
\begin{array}{c}
A_{1,j}^f\\
\vdots\\
A_{k-2,j}^f\\
B_{1,j}^f\\
\vdots\\
B_{k-1,j}^f\\
C_{1,j}^f\\
C_{2,j}^f
\end{array}
\right) .
\]
That is, the entries of $\xi ^f_j$ that correspond to coordinates $U_1,\dots U_{k-2}$ are labelled with $A_1, \dots , A_{k-2}$, the entries that correspond to coordinates $V_1,\dots V_{k-1}$ are labelled with $B_1,\dots , B_{k-2}$, and  
the entries that correspond to coordinates $W_1$ and $W_2$ are labelled with $C_1$ and $C_2$ respectively.

Introducing `dummy' variables in $\underline{U}$ and $\underline{V}$ allows a very succinct description of the liftable vector fields. 
We shall define $U_{k-1}=V_k=0$, $U_k=1$ and $U_r=V_r=0$ for $r\le0$ and for $r>k$.

\begin{theorem}[\cite{danthesis}]
\label{firstfamily}
For $1\leq j \leq k-1$ the vector field given by the following components is liftable over $\varphi _k$:
\begin{eqnarray*}
A_{i,j}^1&=&(k-i)(k-j)U_iU_j   , \qquad 1\le i\le k-2,\\
B_{i,j}^1&=&k\sum_{r=1}^{i-1}U_{i+j-r}V_r-k\sum_{r=1}^iU_rV_{i+j-r}-(i-1)(k-j)U_jV_i\\
& & +\,kV_{i+j}W_1-kU_{i+j}W_2,  \qquad 1\le i\le k-1, \\
C_{1,j}^1&=&k(k-j)U_jW_1,\\
C_{2,j}^1&=&-kV_jW_1+(k-j)U_jW_2.
\end{eqnarray*}
\end{theorem}

\begin{theorem}[\cite{danthesis}]
\label{secondfamily}
For $1\leq j \leq k-1$ the vector field given by the following components is liftable over $\varphi _k$:
\begin{eqnarray*}
A_{i,j}^2&=&-k(k+i-j+1)U_{k+i-j+1}W_1+k\sum_{r=1}^i(k+i-j-2r+1)U_rU_{k+i-j-r+1}\\
& &-j(i+1)U_{i+1}U_{k-j}, \qquad 1\le i\le k-2,\\
B_{i,j}^2&=&-k(k+i-j+1)V_{k+i-j+1}W_1+k\sum_{r=1}^i(k+i-j-r+1)U_rV_{k+i-j-r+1}\\
& &-k\sum_{r=1}^irU_{k+i-j-r+1}V_r-j(i+1)U_{k-j}V_{i+1} , \qquad 1\le i\le k-1, \\
C_{1,j}^2&=&k(k-j+1)U_{k-j+1}W_1+jU_1U_{k-j},\\
C_{2,j}^2&=&k(k-j+1)V_{k-j+1}W_1+jV_1U_{k-j} .
\end{eqnarray*}
\end{theorem}

\begin{theorem}[\cite{danthesis}]
\label{thirdfamily}
For $1\leq j \leq k-1$ the vector field given by the following components is liftable over $\varphi _k$:
\begin{eqnarray*}
A_{i,j}^3&=&-k(k+i-j+1)U_{k+i-j+1}W_2+k\sum_{r=1}^i(k+i-j-r+1)U_{k+i-j-r+1}V_r\\
&&-k\sum_{r=1}^irU_rV_{k+i-j-r+1}-k(i+1)U_{i+1}V_{k-j}, \qquad 1\le i\le k-2,\\
B_{i,j}^3&=&-k(k+i-j+1)V_{k+i-j+1}W_2+k\sum_{r=1}^i(k+i-j-2r+1)V_rV_{k+i-j-r+1}\\
&&-k(i+1)V_{i+1}V_{k-j}, \qquad 1\le i\le k-1, \\
C_{1,j}^3&=&k(k-j+1)U_{k-j+1}W_2+kU_1V_{k-j}\\
C_{2,j}^3&=&k(k-j+1)V_{k-j+1}W_2+kV_1V_{k-j}.
\end{eqnarray*}
\end{theorem}

In theory we can divide the vector fields in the last family by $k$ and still have a nice description, we do not do so to emphasize the similarity of this third family with the second family.

\begin{example}
For $k=2$ the first three liftable vector fields in Example~\ref{eg:wulift} are $\xi _1^1$, $\xi _1^2$, and $\xi _1^3$ respectively, and the final one is $\xi _e$.

It is instructive to evaluate the vector fields for $k=3$ (as well as useful to have them listed somewhere for the purpose of calculating examples).

We find that 
\[
\xi ^1_1 = \left(
\begin{array}{c}
4U_1^2\\
-3U_1V_1+3V_2W_1-3U_2W_2 \\
3U_2V_1-3(U_1V_2+U_2V_1)-2U_1V_2 +3V_3 W_1 - 3U_3W_2 \\
6U_1W_1 \\
-3V_1W_1 + 2U_1 W_2 
\end{array}
\right) . 
\]
However, we have $U_2=V_3=0$ and $U_3=1$ and so one monomial simplifies and some monomials disappear. Therefore, we deduce the following for $\xi _1^1$ and, in a similar way, the descriptions of the other fields:
\[
\xi ^1_1 = \left(
\begin{array}{c}
4U_1^2\\
-3U_1V_1+3V_2W_1 \\
-5U_1V_2 - 3W_2 \\
6U_1W_1 \\
-3V_1W_1 + 2U_1 W_2 
\end{array}
\right) ,
\qquad
\xi ^1_2 =\left(
\begin{array}{c}
0\\
-3U_1V_2 - 3W_2 \\
3V_1 \\
0 \\
-3V_2W_1 
\end{array}
\right) , 
\] 
\[
\xi ^2_1 = \left(
\begin{array}{c}
6U_1\\
-3V_1 \\
-6V_2 \\
9W_1 \\
0 
\end{array}
\right) ,
\qquad
\xi ^2_2 =\left(
\begin{array}{c}
-9W_1 \\
2U_1V_2 \\
-3V_1 \\
2U_1^2\\
6V_2W_1 + 2U_1 V_1 
\end{array}
\right) , 
\] 
\[
\xi ^3_1 = \left(
\begin{array}{c}
9V_1 \\
-6V_2^2 \\
0 \\
9W_2 +3U_1V_2 \\
3V_1V_2
\end{array}
\right) ,
\qquad
\xi ^3_2 =\left(
\begin{array}{c}
-9W_2 - 3U_1 V_2\\
-3V_1 V_2\\
0 \\
3U_1 V_1\\
6V_2W_2 + 3V_1 ^2 
\end{array}
\right) ,
\] 
\[
\xi _e = \left(
\begin{array}{c}
2U_1 \\
2V_1 \\
V_2 \\
3W_1 \\
3W_2 
\end{array}
\right) .
\] 
\end{example}

Now we shall investigate the extent to which the liftable vector fields in these three theorems generate the module of liftable vector fields.
\begin{theorem}
\label{thm:polygenerate}
Let $\varphi _k:(\K ^{2k-2},0)\to (\K ^{2k-1},0)$ be given by the normal form for a corank $1$ minimal stable map of multiplicity $k$. 
Let $\xi $ be a liftable vector field such that its components are polynomials. 

Then, there exist polynomials $g_e$ and $g_{i,j}$ in $\EE _{2k-1}$ such that
\[
\xi = g_e \xi _e + \sum _{i=1}^3 \sum _{j=1}^{k-1} g_{i,j} \xi _j^i  .
\]
\end{theorem}
\begin{proof}
Denote by $e_l$ by the vector $(0,0,\dots, 0,1,0,\dots ,0)^T\in \K ^{2k-1}$ which has zeroes in every position except at position $l$, where it has a $1$.

We apply a negative lexicographical ordering (see \cite{singular} page 14) to the variables in the codomain of $\varphi _k$, i.e., $(U_1, \dots U_{k-2}, V_1, \dots V_{k-1}, W_1, W_2)$. Then the leading terms, denoted $\LT $, of the vector fields are as follows:
\begin{eqnarray*}
\LT (\xi _e) &=& W_2 e _{2k-1} ,\\
\LT (\xi ^1_j) &=& -kW_2 e_{2k-j-2}, {\text{ for }} 1\leq j \leq k-2,\\
\LT (\xi ^2_j) &=&
\left\{
\begin{array}{ll}
k^2W_1 e_{2k-2} , &  {\text{ for }} j=1,\\
-k^2W_1 e_{j-1}, & {\text{ for }} 2\leq j \leq k-2, 
\end{array}
\right. \\
\LT (\xi ^3_j) &=&
\left\{
\begin{array}{ll}
k^2W_2 e_{2k-2} , &  {\text{ for }} j=1,\\
-k^2W_2 e_{j-1}, & {\text{ for }} 2\leq j \leq k-1, 
\end{array}
\right.
\end{eqnarray*}
(The terms in the vector fields may look quadratic but recall that $U_k=1$.)

Applying the Division Algorithm for modules, see \cite{cla} page 202, we can write $\xi $ as 
\[
\xi = g_e \xi _e + \sum _{i,j} g_{ij} \xi ^i_j + r 
\]
where $g_e$ and $g_{ij}$ are polynomials and where $r=0$ or $r$ is a $\K $-linear combination of monomials, none of which is divisible by any of the leading terms of the vector fields.

In particular, using the leading terms for the second and third families of liftable vector fields the $A_i$ and $C_1$ terms of $r$ do not contain any monomials divisible by $W_1$ or $W_2$.
Now, $r$ is obviously liftable as $\xi $ and $g_e \xi _e + \sum _{i,j} g_{ij} \xi ^i_j $ are, so there exist polynomial functions $a_i$ in $\EE _{2k-3}$ and $c$ in $\EE _{2k-2}$ 
such that
\[
\sum _{i=1}^{k-2} a_i(\underline{u},\underline{v}) y^i + c(\underline{u},\underline{v},y) \dfrac{\partial W_1 }{\partial y} = C_1(\underline{u},\underline{v}).
\]
If $c\neq 0$, then the degree in $y$ of $c\dfrac{\partial W_1 }{\partial y}$ is greater than $k-2$. Hence as each $a_i$ contains no $y$'s (as $A_i$ contains no $W_1$ or $W_2$) we cannot find a solution to this equation.

If $c=0$, then any non-zero $a_i$ leads to power of $y$ in the left-hand side of the equation, yet there are obviously none in the right-hand side. Therefore $a_i=c=0$ for all $1\leq i\leq k-2$. Hence $A_i=0$ and $C_1=0$. This just leaves us with the $B_i$, for $1\leq i\leq k-1$, and $C_2$ to determine.

Because of the form of the leading terms in family one and of the Euler vector field we know that the $B_i$ terms and the $C_2$ term of $r$ are not functions of $W_2$ (so we can write $C_2(\underline{U},\underline{V},W_1,W_2)$ as $C_2(\underline{U},\underline{V},W_1)$).
Since $r$ is liftable and $c=0$, there exist polynomial functions $b_i\in \EE _{2k-3}$ so that 
\[
\sum _{i=1}^{k-2} b_i(\underline{u},\underline{v}) y^i  = C_2\left( \underline{u},\underline{v},y^k+\sum_{i=1}^{k-2} u_iy^i \right) .
\]

If the degree of $W_1$ in $C_2$ is $s$, then the right-hand side of the equation is a polynomial in $y$ of degree $sk$.

If $s=0$, then we see by comparing constant terms in the above equation that we must have $b_i=0$ and $C_2=0$.

If $s\geq 1$, then the left-hand side of the equation must have a $y^{sk}$ term (with a function of $\underline{u}$ and $\underline{v}$ as coefficient). However, it is obvious that no such term exists since if $B_i$ has a polynomial in $W_1$ of degree $m$, then the polynomial in $y$ resulting from composition with $\varphi _k$ has a term of degree $mk+i$ where $1\leq i \leq k-1$. So we require $mk+i=sk$, i.e., $i=k(m-s)$. Therefore, $B_i=C_2=0$ for $1\leq i \leq k-1$.

This means that $r=0$ and we can say that $\xi $ is generated by a linear combination of our liftable vector fields.
\end{proof}

Let us now introduce some notation. We shall denote a module generated by elements $z_1,\ldots,z_n$ as $\langle z_1,\ldots ,z_n\rangle $ or $\langle z_i\rangle _{i=1}^n$.

\begin{definition}
Suppose that $V$ is a $\K $-analytic variety defined by the ideal $I(V)=\langle f_1 , \dots , f_q \rangle $ in $(\K ^p,0)$. 
We define
\[
\derlog _0(V)= \{ \xi \in \theta _p  : \, \X( f_ j ) =0 , \ j=1 ,\dots q \} .
\]
\end{definition}
Damon and Mond show in \cite{dm} that for a quasihomogeneous hypersurface we have
\[
\derlog (V) \cong  \langle \X _e \rangle \oplus \derlog _0(V) ,
\]
where $ \langle \X _e \rangle$ is the module in $\theta _p$ generated by $\X _e$. (They show this in the complex analytic case but the same proof holds for real analytic mappings.)

\begin{corollary}
\label{cor:generate}
Let $\varphi _k:(\C ^{2k-2},0)\to (\C ^{2k-1},0)$ be given by the normal form for a corank $1$ minimal stable map of multiplicity $k$ and $V$ be its image. Then,
\[
\derlog (V) = \langle \xi _ e, \X_j^1, \X_j^2, \X_j^3 \rangle _{j=1}^{k-1} .
\]

That is, the module of vector fields liftable over $\varphi _k$ is generated by the vector fields $\X_j^1, \X_j^2, \X_j^3$ for $1\le j\le k-1$, together with the Euler vector field $\X_e$.
\end{corollary}
\begin{proof}
The image of $\varphi _k$ is a hypersurface which we denote by $V$ and by $h_V$ its defining function.
One can explicitly calculate this defining function using the algorithm in Section~2.2 of \cite{mp}. However, we shall use only that the algorithm gives $h_V$ as the determinant of a $k\times k$ matrix which is the sum of a $k\times k$ matrix with entries in the variables $\underline{U}$, $\underline{V}$ and $W_1$ plus the matrix $-W_2I_k$ where $I_k$ is the identity matrix of size $k$.

Since $\xi _e$ is liftable by Proposition~\ref{prop:eulerlift} and complex liftable vector fields are tangent to the image, $\xi _e(h_V)$ is in the ideal generated by $h_V$. By considering the coefficient of $W_2^k$ we can see that $\xi _e (h_V)=k^2h_V$.
Hence it is possible to decompose $\derlog (V)$ into $\langle \X _e \rangle \oplus \derlog _0(V)$ just as Damon and Mond do for quasihomogeneous hypersurfaces.

From the algorithm on pages 613 to 614 in \cite{hm} we know that  $\derlog _0(V)$ can be generated by polynomials. By Theorem~\ref{thm:polygenerate} any polynomial can be given in terms of
$\xi _ e, \X_j^1, \X_j^2, \X_j^3 $ for $1\leq j \leq k-1$ and hence $\derlog _0(V)$ is generated by these liftables. Using the decomposition of $\derlog (V)$ above we deduce the statement in the corollary.
\end{proof}

\begin{conjecture}
\label{conj:derlogv0}
It is natural to conjecture the following:
\begin{enumerate}
\item For $V$ in the preceding corollary the vector fields $\X_j^1, \X_j^2, \X_j^3$ for $1\le j\le k-1$ generate $\derlog _0(V)$. 
\item These liftable vector fields form a Gr\"ober basis.
\end{enumerate}
\end{conjecture}
The first statement is true for $k=2$, see \cite{damonwark} or \cite{west}. The second author has verified the statement via the computer algebra package {\texttt{Singular}} for $k\leq 6$.

Now let us turn to the case of real mappings.
\begin{corollary}
\label{cor:realgenerate}
Let $\varphi _k:(\R ^{2k-2},0)\to (\R ^{2k-1},0)$ be given by the normal form for a corank $1$ minimal stable map of multiplicity $k$ and $V$ be its image.

Then, the module of polynomial vector fields liftable over $\varphi _k$  is generated by the vector fields $\X_j^1, \X_j^2, \X_j^3$ for $1\le j\le k-1$, together with the Euler vector field $\X_e$.
\end{corollary}
\begin{proof}
This follows immediately from Theorem~\ref{thm:polygenerate}
\end{proof}
For more general analytic real vector fields one can conjecture the following.
\begin{conjecture}
The vector fields $\X_j^1, \X_j^2, \X_j^3$ for $1\le j\le k-1$ generate $\derlog _0(V)$ where $V$ is the real part of the image of the complexification of $\varphi _k$.
\end{conjecture}

\section{Proofs of liftability}
The proofs in this section first appeared in \cite{danthesis}.
Proving that a vector field $\X $ is liftable involves writing down a lowerable vector field, i.e., an $\eta $ such that the equation $d\varphi _k \circ \eta = \X \circ \varphi _k $ holds. The liftable vector fields were originally found by calculating them for low values of $k$ via {\texttt{Singular}}, making a guess for the general form and subsequently making an educated case about what the lowerable vector field should be.

The lowerable will be of the form 
\begin{displaymath}\E_j^f=\lt(\begin{array}{c}
a_{1,j}^f\\
\vdots\\
a_{k-2,j}^f\\
b_{1,j}^f\\
\vdots\\
b_{k-1,j}^f\\
c^f
\end{array}\rt)\end{displaymath}
where the $a_i^f$, $b_i^f$ and $c^f$ are functions of the variables $y$ and $u_j$, for $1\leq j \leq k-2$, and $v_j$, for $1\leq j \leq k-2$.

The liftable vector field is of the form
\begin{displaymath}\X_j^f=\lt(\begin{array}{c}
A_{1,j}^f\\
\vdots\\
A_{k-2,j}^f\\
B_{1,j}^f\\
\vdots\\
B_{k-1,j}^f\\
C_1^f\\
C_2^f
\end{array}\rt)\end{displaymath}
where the $A_i^f$, $B_i^f$ and $C_i^f$ are functions of the variables $W_1$, $W_2$, $U_j$, for $1\leq j \leq k-2$, $V_j$, for $1\leq j \leq k-2$. Generally, we shall drop the reference to $f$ where this is suitable.

For the cross cap mapping 
\begin{eqnarray*}
&&\Gp _k(u_1,\ldots ,u_{k-2},v_1,\ldots ,v_{k-1},y)\\
&&\qquad\qquad\qquad\qquad=\lt(u_1,\ldots ,u_{k-2},v_1,\ldots ,v_{k-1},y^k+\sum_{i=1}^{k-2}u_iy^i,\sum_{i=1}^{k-1}v_iy^i\rt)
\end{eqnarray*}
we have the Jacobian matrix 
\begin{displaymath}J_{\Gp _k}=\lt(\begin{array}{c c c c c c c c c}
1 & 0 & \hdots & \hdots & \hdots & \hdots & \hdots & 0 & 0 \\
0 & 1 & \hdots & \hdots & \hdots & \hdots & \hdots & 0 & 0 \\
\vdots & \vdots & \ddots & & & & & \vdots & \vdots  \\
\vdots & \vdots & & \ddots & & & & \vdots & \vdots \\
\vdots & \vdots & & & \ddots & & & \vdots & \vdots \\
\vdots & \vdots & & & & \ddots & & \vdots & \vdots \\
\vdots & \vdots & & & & & \ddots & \vdots & \vdots \\
0 & 0 & \hdots & \hdots & \hdots & \hdots & \hdots & 1 & 0 \\
y & y^2 & \hdots & y^{k-2} & 0 & 0 & \hdots & 0 & \dwi \\
0 & 0 & \hdots & 0 & y & y^2 & \hdots & y^{k-1} & \dwii
\end{array}\rt).\end{displaymath}
From this we can see our equation for liftable vector fields, $d\varphi _k \circ \eta = \X \circ \varphi _k $, gives that $a_i=A_i$ for all $1\leq i \leq k-2$ and $b_i=B_i$ for all $1\leq i \leq k-1$ in the sense that $W_1$ and $W_2$ are functions of $y$ and the other coordinates of the codomain. In all our families this is obviously the case.

The same equation shows that we need to solve two equations to find a liftable and an associated lowerable:
\begin{eqnarray*}
C_1&=&\sum_{i=1}^{k-2}a_iy^i+c\fr{\pr W_1}{\pr y} , \\
C_2&=&\sum_{i=1}^{k-1}b_iy^i+c\fr{\pr W_2}{\pr y} .
\end{eqnarray*}
It is these two equations that we need to verify for each element in each of the three families.

%
%

\begin{lproof}{of Theorem~\ref{firstfamily}}
Instead of proving that $\X _j^1$ is liftable for each $1\leq j\leq k-1$ we shall subtract $(k-j)U_j\X _e$ from $\X _j^1$ to produce a new vector field with simpler entries and show that the new field is liftable.

For all $1\le j\le k-1$ we have
\[
\widetilde{\X }_j^1 = \X _j ^1 - (k-j)U_j\X _e=
\left(
\begin{array}{c}
\widetilde{A}_{1,j}^f\\
\vdots\\
\widetilde{A}_{k-2,j}^f\\
\widetilde{B}_{1,j}^f\\
\vdots\\
\widetilde{B}_{k-1,j}^f\\
\widetilde{C}_{1,j}^f\\
\widetilde{C}_{2,j}^f
\end{array}
\right) 
\]
where
\begin{eqnarray*}
\widetilde{A}_{i,j}^1&=&0,\qquad \text{with }1\le i\le k-2,\\
\widetilde{B}_{i,j}^1&=&k\sum_{r=1}^{i-1}U_{i+j-r}V_r-k\sum_{r=1}^iU_rV_{i+j-r}-(k-1)(k-j)U_jV_i\\
&&\qquad\qquad +\,kV_{i+j}W_1-kU_{i+j}W_2, \qquad \text{with }1\le i\le k-1,\\
\widetilde{C}_{1,j}^1&=&0\\
\widetilde{C}_{2,j}^1&=&-kV_jW_1-(k-1)(k-j)U_jW_2 .
\end{eqnarray*}
It may seem unusual to give the vector field in Theorem~\ref{firstfamily} and not this much simpler one. 
We state the more complicated version because of Conjecture~\ref{conj:derlogv0}(i): the three families generate $\derlog_0(V)$ in the complex analytic case. It is easy to check in low $k$ examples that these new vector fields are not in $\derlog _0(V)$.

For this particular family lowerable vector fields are given by
\begin{displaymath}\E_j^1=\lt(\begin{array}{c}
a_{1,j}^1\\
\vdots\\
a_{k-2,j}^1\\
b_{1,j}^1\\
\vdots\\
b_{k-1,j}^1\\
c_j^1
\end{array}\rt)\end{displaymath}
where
\begin{eqnarray*}
a_{i,j}^1&=&0 , \qquad \text{with }1\le i\le k-2,\\
b_{i,j}^1&=&k\sum_{r=1}^{i-1}u_{i+j-r}v_r-k\sum_{r=1}^iu_rv_{i+j-r}-(k-1)(k-j)u_jv_i+kv_{i+j}W_1-ku_{i+j}W_2\\
&\,&\qquad\qquad\qquad\qquad\qquad\qquad\qquad\qquad\qquad\qquad\qquad\quad\:\text{with }1\le i\le k-1,\\
c_j^1&=&0.
\end{eqnarray*}

With our $\E_j^1$ defined in this way, we will show that composing the Jacobian of our map $\Gp _k$ with $\E_j^1$ gives $\X_j^1$ composed with $\Gp _k$, therefore proving that $\X_j^1$ are liftable vector fields and that $\E_j^1$ are their corresponding lowerable vector fields. Now, in order to make our calculations easier we shall redefine our $W_1$ and $W_2$. We will have
\[W_1=\sum_{i=1}^ku_iy^i\qquad\mbox{and}\qquad W_2=\sum_{i=1}^kv_iy^i.\]
Indeed, this will assist us by adding some symmetry to the calculations, but in order for these definitions to make sense we must introduce extra `dummy' variables, namely $u_k=1$ and $u_{k-1}=v_k=0$. We will have $u_r=v_r=0$ for all $r\le0$ and for all $r>k$. 

It is quite clear, due to the 1's in the Jacobian that $a_{i,j}^1=A_{i,j}^1\circ \Gp _k$ and $b_{i,j}^1=B_{i,j}^1\circ \Gp _k$. This leaves us to show that our expressions for $C_{1,j}^1$ and $C_{2,j}^1$ are as stated. So, firstly, let us calculate $C_{1,j}^1$. 
This is easy as $a_{i,j}=0$ for all $i$ and $j$, and $c=0$, so $C_1=0$.

 Now let us consider $C_{2,j}^1$. We want to show that
\[C_{2,j}^1=-kV_jW_1+(k-j)U_jW_2.\]
To prove this we will start once again by composing the Jacobian with the appropriate terms in the lowerables. We have
\begin{eqnarray*}
&\,&\sum_{i=1}^{k-1}b_{i,j}^1y^i+c\,\dwii\\
&=&\sum_{i=1}^{k-1}b_{i,j}^1y^i\\
&=&\sum_{i=1}^{k-1}\lt(k\sum_{r=1}^{i-1}u_{i+j-r}v_r-k\sum_{r=1}^iu_rv_{i+j-r}-(k-1)(k-j)u_jv_i\rt. \\
&&\qquad\qquad\qquad\qquad\qquad\qquad\qquad+kv_{i+j}W_1-ku_{i+j}W_2\Bigg)y^i\\
&=&k\sum_{i=1}^{k-1}\sum_{r=1}^{i-1}\lt(u_{i+j-r}v_r-u_rv_{i+j-r}\rt)y^i-kv_j\sum_{i=1}^{k-1}u_iy^i+kW_1\sum_{i=1}^{k-1}v_{i+j}y^i\\
&&\qquad\qquad\quad-kW_2\sum_{i=1}^{k-1}u_{i+j}y^i-(k-j)(k-1)u_j\sum_{i=1}^{k-1}v_iy^i .
\end{eqnarray*}
The final term is equal to $-(k-j)(k-1)u_jW_2$ and is part of the claimed $C^1_{2,j}$.
The second term is almost equal to $-kv_jW_1$, which is equal to the other term in $C^1_{2,j}$, all we are missing is $-kv_jy^k$. Hence we are done if we can show the following holds:
\[
k\sum_{i=1}^{k-1}\sum_{r=1}^{i-1}\lt(u_{i+j-r}v_r-u_rv_{i+j-r}\rt)y^i
+kv_jy^k
+kW_1\sum_{i=1}^{k-1}v_{i+j}y^i
-kW_2\sum_{i=1}^{k-1}u_{i+j}y^i
=0.
\]

Let us rearrange the left-hand side and drop the factor $k$ which is now irrelevant. We have
\begin{eqnarray*}
&\,&\sum_{i=1}^{k-1}\sum_{r=1}^{i-1}\lt(u_{i+j-r}v_r-u_rv_{i+j-r}\rt)y^i+W_1\sum_{i=1}^{k-1}v_{i+j}y^i-W_2\sum_{i=1}^{k-1}u_{i+j}y^i+v_jy^k\\
&=&\sum_{i=1}^{k-1}\sum_{r=1}^{i-1}\lt(u_{i+j-r}v_r-u_rv_{i+j-r}\rt)y^i+\sum_{i=1}^{k-1}\sum_{r=1}^k\lt(v_{i+j}u_r-u_{i+j}v_r\rt)y^{i+r}+v_jy^k\\
&=&\sum_{i=1}^{k-1}\sum_{r=1}^{i-1}\lt(u_{i+j-r}v_r-u_rv_{i+j-r}\rt)y^i+\sum_{i=1}^k\sum_{r=1}^{k-1}\lt(v_{r+j}u_i-u_{r+j}v_i\rt)y^{i+r}+v_jy^k\\
&=&\sum_{i=1}^{k-1}\sum_{r=1}^{i-1}\lt(u_{i+j-r}v_r-u_rv_{i+j-r}\rt)y^i-\sum_{i=1}^k\sum_{r=1}^{k-1}\lt(u_{r+j}v_i-u_iv_{r+j}\rt)y^{i+r}+v_jy^k.
\end{eqnarray*}
To show that
\begin{equation}
\label{eq:eq1}
\sum_{i=1}^{k-1}\sum_{r=1}^{i-1}\lt(u_{i+j-r}v_r-u_rv_{i+j-r}\rt)y^i-\sum_{i=1}^k\sum_{r=1}^{k-1}\lt(u_{r+j}v_i-u_iv_{r+j}\rt)y^{i+r}+v_jy^k=0
\end{equation}
we consider coefficients of $y^p$ for different $p$. 

First suppose that $1\le p\le k-1$. Obviously only the double summations will yield coefficients of $y^p$ and the first of these gives 
\[
\sum_{r=1}^{p-1}\lt(u_{p+j-r}v_r-u_rv_{p+j-r}\rt).
\]
The second summation provides a $y^p$ term only if $i\leq p-i$. If we let $r=p-i$, then the summation yields the term
\[
-\sum_{i=1}^{p-1}\lt(u_{p-i+j}v_i-u_iv_{p-i+j}\rt)
\]
and obviously cancels with the previous summation.

Next, let us consider the particular case where $p=k$. Here, the first double summation in Equation~\ref{eq:eq1} yields no coefficient, as $i\leq k-1$. The second, however, does yield coefficients, as does the term $v_jy^k$. We have
\begin{eqnarray*}
&\,&-\sum_{r=1}^{k-1}\lt(u_{k+j-r}v_r-u_rv_{k+j-r}\rt)+v_j\\
&=&-\sum_{r=j}^{k-1}\lt(u_{k+j-r}v_r-u_rv_{k+j-r}\rt)+v_j\\
&\,&\quad\qquad\text{since for $r<j$, the subscripts of $u$ and $v$ are greater than $k$,}\\
&=&-\sum_{r=j+1}^{k-1}\lt(u_{k+j-r}v_r-u_rv_{k+j-r}\rt),
\text{since $u_k=1$ and $v_k$=0,}\\
&=&-\sum_{r=j+1}^{k-1}u_{k+j-r}v_r+\sum_{r=j+1}^{k-1}u_rv_{k+j-r}\\
&=&-\sum_{r=j+1}^{k-1}u_{k+j-r}v_r+\sum_{s=j+1}^{k-1}u_{k+j-s}v_s\\
&\,&\quad\qquad\text{letting $r=k+j-s$ in the second summation}\\
&=&0.
\end{eqnarray*}
Finally we look at the case where $p>k$. Neither the first double summation in Equation~\ref{eq:eq1} nor the last term yield any part of the required coefficient.

In the second double summation the greatest that $i$ can be is $k$, meaning the smallest $r$ can be is $p-k$. On the other hand, the greatest $r$ can be is $k-1$, meaning that the smallest $i$ can be is $p-k+1$. After substituting these limits into the double summation in Equation~\ref{eq:eq1} we actually find that $r$ can only be as large as $k-j$, since, if $r$ were any bigger, then the subscripts of $u$ and $v$ would be greater than $k$, and would yield zeros. This then limits $i$ to be no smaller than $j+p-k$. 
For these reasons, we find that the coefficient of $y^p$ is given by the summation
\begin{eqnarray*}
& &-\sum_{r=p-k}^{k-j}\lt(u_{j+r}v_{p-r}-u_{p-r}v_{j+r}\rt)\\
&=&-\sum_{r=p-k}^{k-j}u_rv_{p+j-r}+\sum_{r=p-k}^{k-j}u_{p+j-r}v_r\\
&=&-\sum_{r=p-k}^{k-j}u_rv_{p+j-r}+\sum_{s=p-k}^{k-j}u_sv_{p+j-s}, \\
& &\text{letting $s=p+j-r$ in the second summation, and changing the limits }\\
&=&0
\end{eqnarray*}
as wanted. 

Hence all coefficients of powers of $y$ in Equation~\ref{eq:eq1} are zero. This then completes the proof that the vector field $\widetilde{\X }^1_j $ is liftable for all $1\leq j \leq k-1$. Since $\X ^1_j$ is a sum of this and a multiple of the Euler vector field, which is liftable, it is liftable as claimed.
\end{lproof}

\begin{lproof}{of Theorems~\ref{secondfamily} and \ref{thirdfamily}}
We treat the second and third families together. One can certainly see that they are similar. 

In both families we need to solve the two equations. These equations are similar and so we introduce new variables and solve a single equation which can then be used to give solutions of both equations.

To this end, let 
\[
X=\sum_{i=1}^kx_iy^i\qquad\mbox{and}\qquad Z=\sum_{i=1}^kz_iy^i
\]
with $x_r=z_r=0$ for $r\le0$ and for $r>k$.

The concept here is that we can let the $x$ variables be $u$ variables or $v$ variables and similarly for $z$ variables. Thus by such a choice we can have $Z=W_1$ or $W_2$. Also, solutions to the equation 
\begin{equation}
\label{eqn:e1}
C=\sum _{i=1}^l \alpha _{i}y^i + c \dfrac{\partial Z }{\partial y } 
\end{equation}
where $C$ is a function of $x$, $z$, $X$ and $Z$,
will provide us with solutions to the two equations.

Thus if we can prove that for each $j$ the following is a solution of the equation we get solutions to the two equations: 
\begin{eqnarray*}
\alpha _{i,j}&=&-k(k+i-j+1)z_{k+i-j+1}X+k\sum_{r=1}^i(k+i-j-r+1)x_rz_{k+i-j-r+1}\\
& &-k\sum_{r=1}^{i+1}rx_{k+i-j-r+1}z_r+(k-j)(i+1)x_kx_{k-j}z_{i+1} \\
c_j&=&k\sum_{r=1}^jx_{k-j+r}y^r+x_{k-j}(k-(k-j)x_k) \\
C&=&k(k-j+1)z_{k-j+1}X+jx_{k-j}z_1+(k-j)(1-x_k)x_{k-j}z_1
\end{eqnarray*}
For the first equation we need $Z=W_1$ while for the second equation we need $Z=W_2$. To get the second family (and satisfy these equations) we take $X=W_1$ while the third family requires $X=W_2$.

Hence, to summarise, for $C_{1,j}^2$ we take $x=u$ and $z=u$, and for $C_{1,j}^3$ we have $x=v$, $z=u$. Meanwhile, for $C_{2,j}^2$ we have $x=u$ and $z=v$, whilst for $C_{2,j}^3$ we take $x=v$ and $z=v$. Simplification will then construct the vector fields specified in the theorem.

We shall now start to substitute our expressions into Equation~\ref{eqn:e1}, and evaluate accordingly. We want to prove
\begin{eqnarray*}
&&\sum_{i=1}^l\lt(-k(k+i-j+1)z_{k+i-j+1}X+g_i\rt)y^i\\
&&\qquad\qquad +\,\lt(k\sum_{r=1}^jx_{k-j+r}y^r+jx_{k-j}+(k-j)(1-x_k)x_{k-j}\rt)\fr{\pr Z}{\pr y}\\
&&\qquad\qquad\qquad\qquad =\,k(k-j+1)z_{k-j+1}X+jx_{k-j}z_1+(k-j)(1-x_k)x_{k-j}
\end{eqnarray*}
where
\begin{eqnarray*}
g_i&=&k\sum_{r=1}^i(k+i-j-r+1)x_rz_{k+i-j-r+1}-k\sum_{r=1}^irx_{k+i-j-r+1}z_r\\
&-&k(i+1)x_{k-j}z_{i+1}+(k-j)(i+1)x_kx_{k-j}z_{i+1}.
\end{eqnarray*}
Thus, we have
\begin{eqnarray*}
&&-\,k\sum_{i=1}^l(k+i-j+1)z_{k+i-j+1}Xy^i+\sum_{i=1}^lg_iy^i+k\sum_{r=1}^jx_{k-j+r}y^r\fr{\pr Z}{\pr y}\\
&&+\,jx_{k-j}\fr{\pr Z}{\pr y}+(k-j)(1-x_k)x_{k-j}\fr{\pr Z}{\pr y}\\
&=&k(k-j+1)z_{k-j+1}X+jx_{k-j}z_1+(k-j)(1-x_k)x_{k-j}.
\end{eqnarray*}
Now, let us simplify one of the terms on the left hand side of this equation. We have
\begin{eqnarray*}
k\sum_{r=1}^jx_{k-j+r}y^r\fr{\pr Z}{\pr y}&=&k\sum_{r=1}^jx_{k-j+r}y^r\sum_{i=1}^kiz_iy^{i-1}\\
&=&k\sum_{r=2}^lx_{k-j+r}y^r\sum_{i=1}^kiz_iy^{i-1}+kx_{k-j+1}y\sum_{i=1}^kiz_iy^{i-1}\\
&=&k\sum_{r=1}^lx_{k-j+r+1}y^{r+1}\sum_{i=1}^kiz_iy^{i-1}+kx_{k-j+1}\sum_{i=1}^kiz_iy^i\\
&=&k\sum_{i=1}^l\sum_{r=1}^krx_{k+i-j+1}y^{i+r}+kx_{k-j+1}\sum_{i=1}^kiz_iy^i.
\end{eqnarray*}
Using this, and substituting in our summations for $X$ and $Z$ then leaves us to show
\begin{eqnarray}
\label{eqn:e2}
&&-\:k\sum_{i=1}^l\sum_{r=1}^k(k+i-j+1)x_rz_{k+i-j+1}y^{i+r}+k\sum_{i=1}^l\sum_{r=1}^krx_{k+i-j+1}z_ry^{i+r}\nonumber\\
&&+\:kx_{k-j+1}\sum_{i=1}^kiz_iy^i+\sum_{i=1}^lg_iy^i+jx_{k-j}\sum_{i=1}^kiz_iy^{i-1}\nonumber\\
&&+\:(k-j)(1-x_k)x_{k-j}\sum_{i=1}^kiz_iy^{i-1}\nonumber\\
&=&k(k-j+1)z_{k-j+1}\sum_{i=1}^kx_iy^i+jx_{k-j}z_1+(k-j)(1-x_k)x_{k-j}.
\end{eqnarray}

To show this equation holds, we shall consider different cases. Firstly, we shall consider powers of $y$ strictly bigger than $k$. On the right hand side of Equation~\ref{eqn:e2}, it is clear that the highest power of $y$ is $k$, so we need to show that all coefficients of powers of $y$ bigger than $k$ on the left hand side of this equation simplify to equal zero. To do this, we shall compare coefficients of $y^{k+q}$, where $q\in\N$. The only terms on the left hand side of Equation~\ref{eqn:e2} that yield powers of $y$ bigger than $k$ are
\[-k\sum_{i=1}^l\sum_{r=1}^k(k+i-j+1)x_rz_{k+i-j+1}y^{i+r}+k\sum_{i=1}^l\sum_{r=1}^krx_{k+i-j+1}z_ry^{i+r}.\]
The highest value $i$ can take is $l$, in which case $r=k+q-l$, since the power $i+r$ must equal $k+q$. If we then substitute $i=l$ and $r=k+q-l$ into the two double summations, we get
\[-k(k-j+l+1)x_{k+q-l}z_{k-j+l+1}+k(k+q-l)x_{k-j+l+1}z_{k+q-l}.\]
The second highest value that $i$ can take is $l-1$, in which case $r=k+q-l+1$. Substituting these values in gives
\[-k(k-j+l)x_{k+q-l+1}z_{k-j+l}+k(k+q-l+1)x_{k-j+l}z_{k+q-l+1}.\]
We can continue reducing $i$ and adjusting $r$ accordingly, so long as $i+r=k+q$. This will continue yielding terms, but we must stop when $r$ reaches $k$ as this is the upper limit of $r$ in each double summation. In that case, $i$ clearly takes the value $q$, and substituting $i=q$ and $r=k$ into the double summations will produce our final term, namely
\[-k(k-j+q+1)x_kz_{k-j+q+1}+k^2x_{k-j+q+1}z_k.\]
If we then take the sum of all these terms, we have
\[-k\sum_{r=k-l+q}^k(2k-j-r+q+1)x_rz_{2k-j-r+q+1}+k\sum_{r=k-l+q}^krx_{2k-j-r+q+1}z_r.\]
Now, if we let $t=2k-j-r+q+1$ in the second summation and adjust the limits accordingly we have
\[-k\sum_{r=k-l+q}^k(2k-j-r+q+1)x_rz_{2k-j-r+q+1}+k\sum_{t=k-j+q+1}^{k-j+l+1}(2k-j-t+q+1)x_tz_{2k-j-t+q+1}.\]
Note that if $r<k-j+q+1$ then $2k-j-r+q+1>2k-j-k+j-q-1+q+1=k$, hence the lower limit of $r$ in the first summation must be $k-j+q+1$, since if $r$ were any lower, then the subscript of $z$ would be bigger than $k$, and these terms would disappear due to the dummy variables. Similarly, the upper limit of $r$ in the second summation must be $k$, else the subscript of $x$ would be bigger than $k$. Thus, we are left with
\begin{eqnarray*}
&&-\:k\sum_{r=k-j+q+1}^k(2k-j-r+q+1)x_rz_{2k-j-r+q+1}\\
&&\qquad\qquad\qquad\qquad+\:k\sum_{t=k-j+q+1}^k(2k-j-t+q+1x_tz_{2k-j-t+q+1}=0
\end{eqnarray*}
as wanted. Hence all coefficients of powers of $y$ bigger than $k$ simplify to equal zero.

We will now consider the powers of $y$ equal to or less than $k$. We must show that coefficients of a general power of $y$, say $y^p$, with $1\le p\le k$ are equal on both sides of Equation~\ref{eqn:e2}. On the right hand side, the only coefficient of $y^p$ is
\[k(k-j+1)x_pz_{k-j+1}.\]
So we want to show that all coefficients of $y^p$ on the left hand side of Equation~\ref{eqn:e2} simplify to give this expression. Comparing coefficients of $y^p$ in the double summations we have
\begin{eqnarray*}
&&-\:k\sum_{r=1}^{p-1}(k-j+p-r+1)x_rz_{k-j+p-r+1}+\sum_{r=1}^{p-1}rx_{k-j+p-r+1}z_r+kpx_{k-j+1}z_p\\
&&+\:g_p+j(p+1)x_{k-j}z_{p+1}+(k-j)(p+1)(1-x_k)x_{k-j}z_{p+1}\\
&=&-k\sum_{r=1}^{p-1}(k-j+p-r+1)x_rz_{k-j+p-r+1}+\sum_{r=1}^{p-1}rx_{k-j+p-r+1}z_r+kpx_{k-j+1}z_p\\
&&+\:k\sum_{r=1}^p(k-j+p-r+1)x_rz_{k-j+p-r+1}-k\sum_{r=1}^prx_{k-j+p-r+1}z_r\\
&&-\:k(p+1)x_{k-j}z_{p+1}+(k-j)(p+1)x_kx_{k-j}z_{p+1}\\
&&+\:j(p+1)x_{k-j}z_{p+1}+(k-j)(p+1)(1-x_k)x_{k-j}z_{p+1}\\
&=&-\:k\sum_{r=1}^{p-1}(k-j+p-r+1)x_rz_{k-j+p-r+1}+\sum_{r=1}^{p-1}rx_{k-j+p-r+1}z_r+kpx_{k-j+1}z_p\\
&&+\:k\sum_{r=1}^{p-1}(k-j+p-r+1)x_rz_{k-j+p-r+1}-k\sum_{r=1}^{p-1}rx_{k-j+p-r+1}z_r\\
&&+\:k(k-j+1)x_pz_{k-j+1}-kpx_{k-j+1}z_p-k(p+1)x_{k-j}z_{p+1}\\
&&+\:(k-j)(p+1)x_kx_{k-j}z_{p+1}+j(p+1)x_{k-j}z_{p+1}\\
&&+\:(k-j)(p+1)(1-x_k)x_{k-j}z_{p+1}\\
&=&kpx_{k-j+1}z_p+k(k-j+1)x_pz_{k-j+1}-kpx_{k-j+1}z_p-k(p+1)x_{k-j}z_{p+1}\\
&&+\:(k-j)(p+1)x_kx_{k-j}z_{p+1}+j(p+1)x_{k-j}z_{p+1}\\
&&+\:(k-j)(p+1)(1-x_k)x_{k-j}z_{p+1}\\
&=&kpx_{k-j+1}z_p+k(k-j+1)x_pz_{k-j+1}-kpx_{k-j+1}z_p\\
&=&k(k-j+1)x_pz_{k-j+1}
\end{eqnarray*}
as wanted. Hence all coefficients of positive powers of $y$ equal to or less than $k$ simplify to give the same expression on both sides of Equation~\ref{eqn:e2}.

The last thing that we need to mention to complete the proof that Equation~\ref{eqn:e2} holds is that the constant terms clearly equate. 
\end{lproof}

%
%

\section{An application to classification of maps}
\label{sec:appl}
In this final section we shall apply the vector fields liftable over the cross cap of multiplicity $k$ to the investigation of mappings that occur generically in one-parameter families of mappings, in particular, maps of $\AA _e$-codimension $1$:
\begin{definition}
\label{defn:aecod}
Suppose that $f:(\K ^n,0)\to (\K ^p,0)$ is a smooth mapping. 
The {\em{$\AA _e$-codimension of $f$}}, denoted $\AA _e-\cod(f)$, is defined to be
 \[
 \AA _e-\cod(f) = \dim _\K \dfrac{\theta (f)}{tf(\theta _n) + wf (\theta _p) } .
 \]
 \end{definition}
We can see that, by definition, $f$ is stable if and only if the $\AA _e$-codimension of $f$ is $0$. A natural task is to find maps with $\AA _e$-codimension $1$. We can do this by relating $\AA _e$-codimension of $f$ to the codimension of a different map.

\begin{definition}[\cite{damoniii,wikatique} and cf.~\cite{br}]
Suppose that $F:(\K^{n'},0)\to(\K^{p'},0)$ is a stable map and $h:(\K^{p'},0)\to(\K^{q},0)$ is an analytic map. Let $V$ be the $\K$-part of the discriminant of the complexification of $F$. 

The {\em{$_V\KK _e$-codimension of $h$}}, denoted $_V\KK _e-\cod (h)$, is defined to be
\[
_V\KK _e-\cod (h)= \dim _\K  \dfrac{\theta (h)}{ \langle \xi (h) \, : \,  \xi \in \derlog(V) \rangle + \langle h_1e_j, h_2e_j , \dots , h_q e_j \rangle _{j=1}^q  } 
\]
where $h=(h_1,h_2, \dots , h_q)$ and, as before,
$e_j =(0,0,\dots, 0,1,0,\dots ,0)^T\in \K ^q$ which has zeroes except at position $j$, where it has a $1$. 
\end{definition}

\begin{example}
\label{eg:vkwu}
Let $F:(\K ^3,0)\to (\K ^4,0)$ be the trivial extension of the Whitney umbrella given by $F(x,v_1,y)=(x,v_1,y^2,v_1y)$. 
The $\K$-part of the image of the complexification of  $F$, denoted $V$, is defined in $(\K ^4,0)$, with coordinates $(X,V_1,W_1,W_2)$,  by $W_2^2-V_1^2W_1=0$. It is easy to calculate that $\derlog (V)$ is generated by the vector fields from Example~\ref{eg:wulift} and the trivial vector field $\partial /\partial X$. (In the case of complex analytic maps we have proved this in Corollary~\ref{cor:generate} but it can be proved for real analytic maps, see \cite{west}.)

Let $h(X,V_1,W_1,W_2)=V_1-p(X,W_1)$.
Then we have
\begin{eqnarray*}
_V \KK _e-\cod (h) &=& \dim _\K \frac{\EE _4 }{\langle W_2, -V_1 +2 W_1 \frac{\partial p}{\partial W_1} , -2W_2 \frac{\partial p}{\partial W_2} ,  V_1 +2 W_1 \frac{\partial p}{\partial W_1} ,\frac{\partial p}{\partial X} \rangle + \langle V_1 - p \rangle } \\
&=& \dim _\K \frac{\EE _4 }{\langle W_2, V_1 , W_1 \frac{\partial p}{\partial W_1} ,\frac{\partial p}{\partial X} , V_1 - p \rangle } \\
&=& \dim _\K \frac{\EE _2 }{\langle W_1 \frac{\partial p}{\partial W_1} ,\frac{\partial p}{\partial X} , p \rangle }.
\end{eqnarray*}
In Exercise~1.1 of \cite{damonwark} there is a similar calculation for the related notion of $\KK _V $-codimension. Note that the above calculation is considerably simpler since the denominator is an ideal rather than a module. 
\end{example}

The two codimensions, $_V \KK $-codimension and $\AA _e$-codimension are intimately connected.

\begin{definition}
Let $F:(\K ^n,0)\to (\K ^p,0)$ be a smooth map and $g:(\K ^r,0)\to (\K ^p,0)$ an immersion which is transverse to $F$, i.e., $dF(T_0(\K ^n,0))+T_0(g(\K^r,0))) =T_0(\K ^p,0)$. The {\em{pullback of $F$ by $g$}}, denoted $g^*(F)$, is the natural map from
\[
(\K ^{r-(p-n)} ,0) \cong \{(x,y)\in (\K^n ,0)\times (\K^r,0) \, : \, F(x)=g(y)\}
\]
to $(\K ^r ,0)$ given by projection on the second factor.
\end{definition}
If $f:(\K ^n,0)\to (\K ^p,0)$ has finite $\AA _e$-codimension, then $f$ can be induced as a pull-pack from a stable map $F$ by an immersion $g$, see \cite{wall}.

The connection between $\AA _e$-codimension and $_V\KK_e$-codimension was given in \cite{damonwark}, Theorem~1, (see also \cite{MM}) and \cite{damoniii}, Lemma~6.2:
\begin{theorem}[\cite{damonwark,MM,damoniii}]
\label{thm:damon}
Suppose that $F:(\K ^n,0)\to (\K ^p,0)$ is a $\K $-analytic stable map and $g:(\K ^r,0)\to (\K ^p,0)$ is an immersion transverse to $F$.
Then, 
\[
\AA _e - \cod (g^*(F) ) = _V\KK _e -\cod (h) 
\]
where $h$ is a submersion $h:(\K ^p,0)\to (\K ^{p-r} ,0)$ such that $h^{-1}(0)$ is the image of $g$ and $V$ is the $\K $-part of the complexification of the discriminant of $F$. 
\end{theorem}
The importance of this theorem is that $_V\KK_e$-codimension is easier to calculate than $\AA _e$-codimension. The denominator in the definition of the latter is not a module whereas it is for the former. This makes calculation easier.

\begin{example}
Consider Example~\ref{eg:vkwu}. If we let $g:(\K ^3,0)\to (\K ^4,0)$ be given by $g(X,W_1,W_2)=(X,p(X,W_1),W_1,W_2)$, then the image of $g$ is equal to $h^{-1}(0)$ and, one calculates, $g$ is transverse to $F$. The pull-back of $F$ by $g$ is a map of the form $f(x,y)=(x,y^2,yp(x,y^2))$. 

Using the example and the previous theorem we calculate that
\[
\AA _e-cod (f) = \dim _\K \frac{\EE _2 }{\langle y \frac{\partial p}{\partial y} ,\frac{\partial p}{\partial x} , p \rangle } 
\]
which is the formula found in \cite{mondr2tor3}. However, it can be seen that the work in calculating the $\AA _e$-codimension of the map $f$ is greatly reduced with this method. 
\end{example}

Let us now look at linear maps $h:(\K ^{2k-1},0)\to (\K ,0)$ on the image $V$ of the minimal cross cap mapping of multiplicity $k$ and find those with $_V\KK _e$-codimension equal to $1$. This allows us to find $\AA _e$-codimension $1$ maps because we can parametrize the zero-set of $h$ by a map $g$ and then the pull-back of the cross cap by $g$ will have $\AA _e$-codimension~$1$.

As usual let $(\underline{U}, \underline{V},\underline{W})=(U_1,\ldots,U_{k-2},V_1,\ldots,V_{k-1},W_1,W_2)$ denote the coordinates on the codomain of the minimal cross cap of multiplicity $k$. Now consider a general linear function $h$ on $\K ^{2k-1}$:
 \[
h(U,V,W)=\sum_{i=1}^{k-2}\Ga_iU_i+\sum_{i=1}^{k-1}\Gb_iV_i+\Gg_1W_1+\Gg_2W_2,
\]
where the coefficients $\alpha $, $\beta $ and $\gamma $ are elements of $\K $.

For $_V \KK _e$-codimension $1$ functions the denominator in the definition of $_V \KK $-codimension is equal to the maximal ideal, denoted $\M $, in $\theta (h)= \EE _{2k-1}$. In the following we work modulo $\M ^2$ as later we shall simplify calculations by using Nakayama's Lemma.

A different version of the following was first proved in Section~4.4 of \cite{danthesis}.
\begin{theorem}
\label{thm:modulom}
Let $h:(\K ^{2k-1},0)\to (\K ,0)$ be the linear map above and $\xi ^i_j$, where $1\leq i\leq 3$ and $1\leq j \leq k-1$, denote the vector fields liftable over $\varphi _k$ in Theorems~\ref{firstfamily}, \ref{secondfamily} and \ref{thirdfamily}.

Then,
\begin{eqnarray*}
\X_j^1(h)&=&k \left( - \Gb_{k-j}W_2 + \sum_{i=k-j+1}^{k-1}\Gb_iV_{i+j-k}  \right) \mod \M^2, \\
\X_j^2(h)&=&\sum_{i=j-1}^{k-2}\Ga _i (k-i+j-1)U_{i-j+1} -k \sum_{i=j-1}^{k-1} (i-j+1) \Gb_{i}V_{i-j+1} \\
& &\qquad -k(k-j+1) \Gg_1 U_{k-j+1}W_1-k^2\Ga_{j-1}W_1 \mod \M^2, \\
\X_j^3(h)&=& k(k-j+1) \Gg_1U_{k+j-1}W_2-k^2\Ga_{j-1}W_2  + k^2 \sum_{i=j}^{k-2}\Ga_i V_{i-j+1} \mod \M^2,
\end{eqnarray*}
for all $1\leq j\leq k-1$.
\end{theorem}
\begin{proof}
For all three families we can ignore any quadratic terms in the description of the vector field. Note that this means we only have to look at terms that include $U_k$.

For $\xi ^1_j$ we have that $A^1_{i,j}$, $C^1_{1,j} $ and $C^1_{2,j}$ are all in $\M ^2$ and hence 
\[
A^1_{i,j} = C^1_{1,j} = C^1_{2,j} = 0 \mod \M ^2 .
\]
For $B^1_{i,j}$ the second summation and the second and third to last terms are in $\M ^2$ and so 
\[
B^1_{i,j} = k \sum _{r=1}^{i-1} U_{i+j-r} V_r - k U_{i+j} W_2 \mod \M ^2 .
\] 
Therefore we have
\[
\xi ^1_j (h) = \sum _{i=1}^{k-1} \beta _i \left( k \sum _{r=1}^{i-1} U_{i+j-r} V_r - k U_{i+j} W_2 \right) \mod \M ^2.
\]
Now, $U_{i+j-r} = 1$ if $i+j-r=k$, i.e., $r=i+j-k$, and similarly $U_{i+j} =1$ if $i=k-j$. Thus 
\[
\xi ^1_j (h) = k \sum _{i=1}^{k-1} \beta _i V_{i+j-k}  - k \beta _{k-j}  W_2 \mod \M ^2.
\]
Now, if $i\leq k-j$, then $V_{i+j-k}=0$ and so we can replace the lower limit in the summation by $k-j+1$.

For the second family, $\xi ^2_j$, we note that 
\begin{eqnarray*}
A^2_{i,j} &=& - k(k+i-j+1) U_{k+i-j+1} W_1 + k \sum_{r=1}^i (k+1-j-2r+1) U_r U_{k+i-j-r+1} \mod \M ^2 , \\
B^2_{i,j} &=& -k \sum _{r=1}^i r U_{k+i-j-r+1}V_r \mod \M ^2, \\
C^2_{1,j} &=& k (k-j+1)U_{k-j+1}W_1 \mod \M ^2, \\
C^2_{2,j} &=& 0 \mod \M ^2 .
\end{eqnarray*}
Similar reasoning as for $\xi ^1_j(h)$ involving $U_k=1$ gives the required form of $\xi ^2_j(h)$. (This time we use $i\leq j-1$ to get that $U_{i-j+1}=V_{i-j+1}=0$ and hence can alter the lower limit for the summations.)

For $\xi ^3_j$ we have
\begin{eqnarray*}
A^3_{i,j} &=& - (k+i-j+1) U_{k+i-j+1} W_2 + k^2 V_{i-j+1} \mod \M ^2 ,\\
B^3_{i,j} &=& 0 \mod \M ^2 ,\\
C^3_{1,j} &=& k(k-j+1) U_{k-j+1} W_2 \mod \M ^2 , \\
C^3_{2,j} &=& 0 \mod \M ^2 .
\end{eqnarray*}
From this we can deduce the stated form of $\xi ^3_j(h)$. 
\end{proof}

The following is Theorem~4.4.7 of \cite{danthesis}.
\begin{theorem}
\label{thm:mainlinear}
For $k>2$ suppose that $\alpha _{k-2}$ and $\beta _{k-1}$ are non-zero. Then,
\[
\langle \X ^1_j (h) ,  \X ^2_j (h),  \X ^3_j (h) , h \rangle _{j=1}^{k-1} = \M .
\]
\end{theorem}
\begin{proof}
Let $I$ be an ideal in $\EE _{2k-1}$. By Nakayma's Lemma we have that $I=\M $ if and only if $I+\M^2 = \M$, (see \cite{wall} Lemma~2, page 929).
We shall have $I= \langle \X ^1_j (h) ,  \X ^2_j (h),  \X ^3_j (h) , h \rangle _{j=1}^{k-1}$.

Consider the elements of $I$ generated by $\X ^1_j (h)$. 
By Theorem~\ref{thm:modulom}  we can see that (modulo $\M ^2$) these elements are $\sum_{i=1}^{j-1}\Gb_{k-j+i}V_i-\Gb_{k-j}W_2$. They can be better understood if we write them as the product of two matrices:
\begin{displaymath}\lt(\begin{array}{c c c c c c c c c}
0 & \hdots & \hdots & \hdots & \hdots & \hdots & \hdots & 0 & -\Gb_{k-1}\\
\Gb_{k-1} & 0 & \hdots & \hdots & \hdots & \hdots & \hdots & 0 & -\Gb_{k-2}\\
\Gb_{k-2} & \Gb_{k-1} & 0 & \hdots & \hdots & \hdots & \hdots & 0 & -\Gb_{k-3}\\
\vdots & \vdots & \vdots & \ddots & & & & \vdots & \vdots\\
\vdots & \vdots & \vdots & & & \ddots & & \vdots & \vdots\\
\Gb_3 & \Gb_4 & \Gb_5 & \hdots & \hdots & \hdots & \hdots & 0 & -\Gb_2\\
\Gb_2 & \Gb_3 & \Gb_4 & \hdots & \hdots & \hdots & \hdots & \Gb_{k-1} & -\Gb_1
\end{array}\rt)
\lt(\begin{array}{c}
V_1\\ V_2\\ V_3\\ \vdots\\ \vdots\\ V_{k-2}\\ W_2
\end{array}\rt).
\end{displaymath}
Written in this form it is easy to see that if $\beta _{k-1}$ is non zero, then 
\[
\langle \X ^1_j (h) \rangle _{j=1}^{k-1}+\M ^2 = \langle W_2, V_1 , V_2 , \dots, V_{k-2}\rangle + \M ^2 .
\]
Note that we have not yet concluded that $V_{k-1}$ is in our ideal. 

Consider the elements $\X ^3_j (h)$, $1\leq j \leq k-1$. We can see from Theorem~\ref{thm:modulom} (again modulo $\M ^2$) that these are dependent on $W_2, V_1 , V_2 , \dots, V_{k-2} $. The variable $V_{k-1}$ does not actually occur with non-zero coefficient. Thus
\[
\langle \X ^1_j (h) ,  \X ^3_j (h) \rangle _{j=1}^{k-1} +\M ^2 =
\langle \X ^1_j (h)  \rangle _{j=1}^{k-1} +\M ^2 .
\]
(Therefore, in some sense, the third family is superfluous to our calculations.)

Turning to the second family, modulo $\langle \X ^1_j (h)  \rangle _{j=1}^{k-1} +\M ^2$ we can write $\X ^2_j (h)$ for $2\leq j\leq k-1$ (in particular $k>2$) as the product of two matrices: 
\[
k
\lt(\begin{array}{c c c c c c c c c}
0 & \hdots & \hdots & \hdots & \hdots & \hdots & \hdots & 0 & -k\Ga_{k-2}\\
(k-1)\Ga_{k-2} & 0 & \hdots & \hdots & \hdots & \hdots & \hdots & 0 & -k\Ga_{k-3}\\
(k-1)\Ga_{k-3} & (k-2)\Ga_{k-2} & 0 & \hdots & \hdots & \hdots & \hdots & 0 & -k\Ga_{k-4}\\
\vdots & \vdots & \vdots & \ddots & & & & \vdots & \vdots\\
\vdots & \vdots & \vdots & & & \ddots & & \vdots & \vdots\\
(k-1)\Ga_{3} & (k-2)\Ga_4 & (k-3)\Ga_5 & \hdots & \hdots & \hdots & \hdots & 0 & -k\Ga_2\\
(k-1)\Ga_{2} & (k-2)\Ga_3 & (k-3)\Ga_4 & \hdots & \hdots & \hdots & \hdots & 3\Ga_{k-2} & -k\Ga_1
\end{array}\rt)
\lt(\begin{array}{c}
U_{k-3}\\ U_{k-4}\\ \vdots\\ \vdots\\ U_2 \\ U_1 \\ W_2
\end{array}\rt).
\]
Since $\X^2_{k-1}(h)=-k^2\alpha _{k-2} W_1-k\beta _{k-1} V_1\mod \M^2 $ we deduce that $W_1$ is in our ideal. With the elements written in the form of a matrix we can see easily that if $\alpha _{k-2}$ is non-zero, then $U_1, \dots , U_{k-3}$ are in the ideal $I+\M^2 $. 

Using 
\[
\xi ^2_1(h)=2k\alpha _{k-1} U_{k-2}-k(k-1)\beta_{k-1} V_{k-1} \mod \langle U_1, \dots ,U_{k-3} ,V_1, \dots ,V_{k-2} ,W_1,W_2 \rangle 
\]
and 
\[
h= \alpha _{k-2}U_{k-2}+\beta _{k-1} V_{k-1} \mod   \langle U_1, \dots ,U_{k-3} ,V_1, \dots ,V_{k-2} ,W_1,W_2 \rangle 
\]
it is easy to deduce that $U_{k-2}$ and $V_{k-1}$ are also in our ideal. Thus 
\[
\langle \X ^1_j (h) ,  \X ^2_j (h),  \X ^3_j (h) , h \rangle _{j=1}^{k-1} +\M ^2 = \M 
\]
as required.
\end{proof}

The following is Theorem~4.5.1 of \cite{danthesis}.
\begin{corollary}
\label{cor:aecod1}
Let $\varphi _k:(\C ^{2k-1},0) \to (\C ^{2k-1},0)$ be the complex minimal cross cap of multiplicity $k>2$ and let $h$ be the general linear map above.  
Suppose that $\alpha _{k-2}$ and $\beta _{k-1}$ are non-zero and that $g$ is an immersion parametrizing $h^{-1}(0)$. Then,
\[
\AA _e - \cod (g^*(\varphi _k)) = 1.
\]
\end{corollary}
\begin{proof}
It is easy to calulate that $g$ is transverse to $\varphi $. 
By Theorem~\ref{thm:damon} we know that
\[
\AA _e{\text{-cod }}(g^*(\varphi _k)) =_V\KK _e-\cod (h) 
=\dim _\K  \dfrac{\EE _{2k-1}}{\langle  \X ^1_j (h) ,  \X ^2_j (h),  \X ^3_j (h) , \xi _e(h),  h \rangle _{j=1}^{k-1} }.
\] 
Since $\xi _e(h)\in \M $ we can see from Theorem~\ref{thm:mainlinear} that 
\[
\langle  \X ^1_j (h) ,  \X ^2_j (h),  \X ^3_j (h) , \xi _e(h),  h \rangle _{j=1}^{k-1} = \M ,
\]
from which the result follows.
\end{proof}

\begin{remarks}
\begin{enumerate}
\item 
In the proof of Theorem~7.3 of \cite{topaug}, following the ideas in \cite{damon1,cplxgen}, it is proved that any sufficiently generic linear function on the minimal cross cap induces an $\AA _e$-codimension $1$ map-germ. Here, the result is improved: we have precise conditions for genericity.
 
\item Note that in the proof of the corollary the Euler vector field, in some sense, plays no part in the fact that the pull-back map is $\AA _e$-codimension $1$. Note that the third family plays no part either. This will be explored in a subsequent paper.
\end{enumerate}
\end{remarks}

\end{document}